\documentclass[preprint,12pt,3p]{elsarticle}
\usepackage{verbatim}
\usepackage[active,new,old]{correct} 
\usepackage{amsmath,amsthm,amsfonts,amssymb}
\usepackage{algorithm}
\usepackage{algpseudocode}
\usepackage{mathdots}
\usepackage{cases}
\usepackage[hidelinks]{hyperref}
\usepackage{graphics}
\usepackage{graphicx}
\usepackage{caption}
\usepackage{subcaption}
\usepackage{multicol}
\usepackage[normalem]{ulem}
\usepackage{xcolor,sectsty}
\definecolor{astral}{RGB}{46,116,181}
\subsectionfont{\color{astral}}
\sectionfont{\color{astral}}
\newtheorem{theorem}{Theorem}[section]

\newtheorem{remark}[theorem]{Remark}

\newtheorem{proposition}[theorem]{Proposition}
\definecolor{darkslategray}{rgb}{0.18, 0.31, 0.31}
\definecolor{warmblack}{rgb}{0.0, 0.26, 0.26}
\definecolor{thrdfc}{rgb}{0.36, 0.54, 0.66}
\definecolor{bole}{rgb}{0.55, 0.71, 0.0}

 \journal{...}

\newcommand{\mb}{\mathbb}

\begin{document}
                                   \begin{frontmatter}
\title{\textcolor{warmblack}{\bf Tight bounds of norms of Wasserstein metric matrix}}
\author{Aaisha Be$^{\dagger, a}$, Nachiketa Mishra{$^{\ddagger, b}$}, Debasisha Mishra{$^{\dagger, c}$}}
\address{$^{\dagger}$Department of Mathematics,\\
 National Institute of Technology Raipur,\\
 Raipur, Chhattisgarh, India\\
$^{\ddagger}$Department of Mathematics,\\
 Indian Institute of Information Technology, \\
 Design and Manufacturing, Kancheepuram,\\
 Chennai, Tamil Nadu, India\\
 \textit{E-mail$^a$}: \texttt{aaishasaeed7\symbol{'100}gmail.com}\\
\textit{E-mail$^b$}: \texttt{nmishra\symbol{'100}iiitdm.ac.in}\\
\textit{E-mail$^c$}: \texttt{dmishra\symbol{'100}nitrr.ac.in.}}
                          
                          \begin{abstract}
Very recently, Bai [Linear Algebra Appl., 681:150-186, 2024 \& Appl. Math. Lett., 166:109510, 2025] studied some concrete structures, and obtained essential algebraic and computational properties of the one-dimensional, two-dimensional and generalized Wasserstein-1 metric matrices. This article further studies some algebraic and computational properties of these classes of matrices. Specifically, it provides lower and upper bounds on the $1,2,\infty$-norms of these matrices, their inverses, and their condition numbers. For the $1$ and $\infty$-norms, these upper bounds are much sharper than the existing ones established in the above mentioned articles.
These results are also illustrated using graphs, and the computation of the bounds is presented in tables for various matrix sizes. It also finds regions for the inclusion of the eigenvalues and numerical ranges of these matrices. A new decomposition of the Wasserstein matrix in terms of the Hadamard product can also be seen. Also, a  decomposition of the Hadamard inverse of the Wasserstein matrix is obtained. Finally, a few bounds on the condition number are obtained using numerical radius inequalities.   
\end{abstract}
                            \begin{keyword}
Wasserstein metric, positive definite matrix, matrix norms, eigenvalues, numerical range.

\vspace{0.2cm}
\end{keyword}
\end{frontmatter}
                              \section{Introduction}\label{sec:intro}
                              The {\it Wasserstein metric} is a measure of the distance between two probability distributions by determining the minimum cost of transforming one distribution into the other.  It is also known as the {\it earth mover's distance}.
In the theory of optimal transport, the Wasserstein distance measures the minimal cost required to change one probability distribution into another, where cost is defined in terms of both the quantity of mass to be moved and the distance over which it must be transported. This metric plays a pivotal role in numerous applications in many fields such as training of generative adversarial networks (GANs), pattern recognition, image processing, machine learning, optimal transport problem, coding theory, seismology \cite{arjovsky2017, blanchet2019,  engquist2013arxiv, kolouri2016, kolouri2017, ling2007, rubner2000, solomon2015}, etc. Bhatia {\it{et al.}} \cite{bhatia2019.EM, bhatia2019.LAA} also studied some fundamental properties of the Wasserstein distance and the Wasserstein mean of two Hermitian positive definite matrices from the perspective of matrix analysis.

\subsection{Motivation}
Despite its theoretical appeal, computing the Wasserstein distance efficiently and accurately remains a significant challenge. To overcome this, the Sinkhorn algorithm, which is an iterative method, is implemented in many problems \cite{cuturi2013, liao2022, liao2024, sinkhorn1967, sinkhorn1974} due to its efficient numerical stability and faster convergence through iterative updates. In this algorithm, the crucial component is the multiplication of the kernel matrix \cite{liao2022, liao2024} or the Wasserstein metric matrix \cite{bai2024} with a vector. So, the problem of solving this matrix-vector product at the minimum cost is required \cite{liao2022}. This is a connection between Wasserstein distance in optimal transport theory and  matrix analysis. Before we start our discussion on how and where this Wasserstein metric matrix is coming from, let us explore the Wasserstein distance and how they are connected to matrix theory.\\ 

\noindent Let $u$ and $v$ be two probability distributions defined on a metric space $(X, d)$.
The \textit{Wasserstein distance} of order $p$ between $u$ and $v$ is defined by the formula.
\begin{equation}
    W_p(u, v) = \left( \inf_{\gamma \in \Gamma(u, v)} \int_{X \times X} d(x, y)^p \, d\gamma(x, y) \right)^{1/p}
\end{equation}
Here $\Gamma(u, v)$ is the set of all joint distributions with marginals $u$ and $v$, $d(s, t)$ is the distance between points $s$ and $t$ and, $\gamma$ is joint probability measure on $X\times X$ . For $p = 2$, the Wasserstein distance metric is equivalently represented as distance between two PSD matrices in $\mathbb{P}(n)$, a set of all PSD matrices of size $n \times n$ is established in \cite{bhatia2019.EM, bhatia2019.LAA, hayoung24}. In the last few years, researchers have explored such Wasserstein distance metric related to Wasserstein Riemannian metric, and also studied on Wasserstein mean and Wasserstein barycentre. But, for $p=1$, we have the Wasserstein-1 metric, and our present work is based on matrix appears while solving the following discretized version,
$$W_\varepsilon(u^h,v^h)=\displaystyle\inf_{\gamma_{ij}}\displaystyle \sum_{i=1}^{n}\sum_{j=1}^{n}[\gamma_{ij}|i-j|h+\varepsilon \gamma_{ij} \ln(\gamma_{ij})],$$
s.t. $$\sum_{i=1}^{n}\gamma_{ij}=u_{i},\; \sum_{j=1}^{n}\gamma_{ij}=v_{j},\; \gamma_{ij}\geq0,\; i,j=1,2,\dots,n,$$
where $$u^h=(u_1,u_2,\dots,u_n)^T~\text{and}~v^h=(v_1,v_2,\dots,v_n)^T$$ 
are discrete distributions derived from the density functions $u(x)$ and $v(y)$ defined on $\Omega\subset \mb{R}$, $h$ is the grid size on the uniform meshgrid, of the entropy-regularized Wasserstein-$1$ metric Kantorovich formulation,
$$W_\varepsilon(u,v)=\displaystyle\inf_{\gamma(x,y)\in\Gamma}\int_{\Omega \times \Omega}|x-y|\gamma(x,y)dxdy+\varepsilon \gamma(x,y) \ln(\gamma(x,y))dxdy,$$
$$\Gamma=\Big\{\gamma(x,y):\;\int_{\Omega}\gamma(x,y)dy=u(x),\;\int_{\Omega}\gamma(x,y)dx=v(y)\Big\}$$
using the Lagrangian multiplier method. Bai \cite{bai2024} defined the {\it kernel matrix} \cite{liao2022} as follows:
$$Q=(q_{ij})\in\mb{R}^{n\times n},~\text{with}~q_{ij}=\lambda^{|i-j|},~~i,j=1,2,\dots,n,$$
where $\lambda=e^{-h/\varepsilon},$ $\varepsilon>0$ is the regularization parameter. And, he named it as the {\it Wasserstein metric matrix}. 
                       \subsection{Related works and problem}
Bai \cite{bai2024, bai2025} studied some concrete structures, and essential algebraic and computational properties of the one-dimensional, two-dimensional and generalized Wasserstein-$1$ metric matrices on the matrix level.
In 2024, Bai \cite{bai2024} provided two representations of a Wasserstein-$1$ metric matrix and proved that it is a positive, Toeplitz, and symmetric matrix. He estimated upper bounds for their $1$ and $\infty$-norms. Under some restrictions on the regularization parameter and the discrete gridsize, he also proved nonsingularity of these matrices and found upper bounds for $1$ and $\infty$-norms of their inverses and condition numbers. After that, in 2025, Bai \cite{bai2025} removed those restrictions and again proved the nonsigularity by providing a decomposition of these matrices and estimated the upper bounds for their inverses and condition numbers with respect to  $1$ and $\infty$-norms. Now, we recall these results in combined form by understanding these matrices.
The explicit form of the one-dimensional Wasserstein-$1$ metric matrix \cite{bai2024} is
                                    \begin{equation}\label{eq:exp Q L}
Q=\begin{bmatrix}
 1 & \lambda & \lambda^2 & \dots & \lambda ^{n-1}\\
     \lambda & 1 & \lambda & \dots & \lambda^{n-2}\\
     \lambda^2 & \lambda & 1 & \dots & \lambda^{n-3}\\
     \vdots & \vdots & \vdots & \dots & \vdots \\
     \lambda ^{n-1} & \lambda ^{n-2} &\lambda ^{n-3} & \dots & 1
     \end{bmatrix}.
     \end{equation}
Setting $L=\begin{bmatrix}
     0 & 0 & 0 & \dots & 0& 0\\
     \lambda & 0 & 0 & \dots& 0 & 0\\
     \lambda^2 & \lambda & 0 & \dots & 0& 0\\
     \vdots & \vdots & \vdots & \ddots & \vdots &\vdots\\
     \lambda ^{n-1} & \lambda ^{n-2} &\lambda ^{n-3} & \dots & \lambda& 0
     \end{bmatrix}$, we have $Q=I+L+L^T$.  
Bai \cite{bai2025} also proved that 
                                      \begin{eqnarray}
     Q&=&(I-\lambda N)^{-1}+(I-\lambda N^T)^{-1}-I,\label{eq:exp Q N}\\
     (I-\lambda N^T)Q(I-\lambda N)&=&\hat{D}_{\lambda},\label{eq:exp Q D}\\
      (I-\lambda N)^{-1} Q^{-1}(I-\lambda N^T)^{-1}&=&\hat{D}_{\lambda}^{-1}\label{eq:Q inv D}\\
Q^{-1}&=&(I-\lambda N)\hat{D}_{\lambda}^{-1}(I-\lambda N^T)\label{eq:Q inv D2}
\end{eqnarray}
where 
     $$N=\begin{bmatrix}
         0 &0 &0 &\dots &0 &0\\
         1 &0 &0 &\dots &0 &0\\
         0 &1 &0 &\dots &0 &0\\
        \vdots & \vdots & \vdots & \ddots & \vdots & \vdots\\
        0 &0 &0 &\dots &0 &0\\
         0 &0 &0 &\dots &1 &0\\
     \end{bmatrix},~\hat{D}_{\lambda}= \begin{bmatrix}
         {1-\lambda^2} &0 &0 &\dots &0 &0\\
         0 &{1-\lambda^2} &0 &\dots &0 &0\\
         0 &0 &{1-\lambda^2} &\dots &0 &0\\
        \vdots & \vdots & \vdots & \ddots & \vdots & \vdots\\
        0 &0 &0 &\dots &{1-\lambda^2} &0\\
         0 &0 &0 &\dots &0 &1\\
     \end{bmatrix}~\text{and}~
     $$    
     $$
     I-\lambda N=\begin{bmatrix}
         1 &0 &0 &\dots &0 &0\\
         -\lambda &1 &0 &\dots &0 &0\\
         0 &-\lambda &1 &\dots &0 &0\\
        \vdots & \vdots & \vdots & \ddots & \vdots & \vdots\\
        0 &0 &0 &\dots &1 &0\\
         0 &0 &0 &\dots &-\lambda &1\\
     \end{bmatrix}.$$
After constructing $B\in \mb{R}^{n \times n}$ as 
 $$B=\begin{bmatrix}
     1 & 0 & 0 & \dots & 0\\
     \lambda & 1 & 0 & \dots & 0\\
     \lambda^2 & \lambda & 1 & \dots & 0\\
     \vdots & \vdots & \vdots & \dots & \vdots \\
     \lambda ^{n-1} & \lambda ^{n-2} &\lambda ^{n-3} & \dots & 1
     \end{bmatrix},$$ we observe that 
     $(I-\lambda N)B=B (I-\lambda N)=I.$
     Therefore, we have $(I-\lambda N)^{-1}=B.$
\begin{proposition}(Inverse of $A=I - aN$)\label{prop:inv}\\
If $A=\begin{bmatrix}
         1 &0 &0 &\dots &0 &0\\
         -a &1 &0 &\dots &0 &0\\
         0 &-a &1 &\dots &0 &0\\
        \vdots & \vdots & \vdots & \ddots & \vdots & \vdots\\
        0 &0 &0 &\dots &1 &0\\
         0 &0 &0 &\dots &-a &1\\
     \end{bmatrix},$ $a\in \mb{R}$, then $A^{-1}=\begin{bmatrix}
     1 & 0 & 0 & \dots & 0\\
     a & 1 & 0 & \dots & 0\\
     a^2 & a & 1 & \dots & 0\\
     \vdots & \vdots & \vdots & \dots & \vdots \\
     a ^{n-1} & a ^{n-2} & a^{n-3} & \dots & 1
     \end{bmatrix}.$    
\end{proposition}
Some results about the one-dimensional Wasserstein-$1$ metric matrices are recalled here.
                            \begin{theorem}(Theorem $2.1$, \cite{bai2024} \& Theorem $2.2$, \cite{bai2025})\label{thm:wass one dim norm bounds}\\
 For the one-dimensional Wasserstein-$1$ metric  $Q\in \mathbb{R}^{n\times n},$ the assertions listed below are valid:
 \begin{enumerate}[(i)]
     \item $Q$ is positive, symmetric, Toeplitz, and positive definite.
     \item $ \|Q\|_{\infty}=\|Q\|_1 <\frac{1+\lambda}{1-\lambda}.$
     \item  $Q$ is invertible and satisfy     
     $\|Q^{-1}\|_{\infty}=\|Q^{-1}\|_1\leq \frac{1+\lambda}{1-\lambda}.$
     \item $\kappa_{\infty}(Q)=\kappa_1(Q)<\frac{(1+\lambda)^2}{(1-\lambda)^2}.$
 \end{enumerate}
\end{theorem}
We now recall the two-dimensional Wasserstein-$1$ metric matrices \cite{bai2024}.
The two-dimensional  Wasserstein-$1$ metric matrix $Q\in \mathbb{R}^{nm\times nm}$ is defined as the Kronecker product between the matrices $Q_{[2]}\in \mathbb{R}^{m\times m}$ and $Q_{[1]}\in \mathbb{R}^{n\times n}$, i.e.,
     $$Q=Q_{[2]}\otimes Q_{[1]},$$ 
 where 
     $$Q_{[1]}=\begin{bmatrix}
     1 & \lambda_1 & \lambda_1^2 & \dots & \lambda_1 ^{n-1}\\
     \lambda_1 & 1 & \lambda_1 & \dots & \lambda_1^{n-2}\\
     \lambda_1^2 & \lambda_1 & 1 & \dots & \lambda_1^{n-3}\\
     \vdots & \vdots & \vdots & \dots & \vdots \\
     \lambda_1^{n-1} & \lambda_1 ^{n-2} &\lambda_1^{n-3} & \dots & 1
     \end{bmatrix}
     ~ \text{and}~ 
     Q_{[2]}=\begin{bmatrix}
     1 & \lambda_2 & \lambda_2^2 & \dots & \lambda_2 ^{m-1}\\
     \lambda_2 & 1 & \lambda_2 & \dots & \lambda_2^{m-2}\\
     \lambda_2^2 & \lambda_2 & 1 & \dots & \lambda_2^{m-3}\\
     \vdots & \vdots & \vdots & \dots & \vdots \\
     \lambda_2 ^{m-1} & \lambda_2 ^{m-2} &\lambda_2 ^{m-3} & \dots & 1
     \end{bmatrix}
     $$
are the one-dimensional Wasserstein-$1$ metric matrices. Here, we recall some existing results on this matrix.
                                \begin{theorem}(Theorem $3.1$, \cite{bai2024} \& Theorem $3.2$, \cite{bai2025})\label{thm:wass two dim norm bounds}\\
 For the two-dimensional Wasserstein-$1$ metric matrix $Q\in \mathbb{R}^{nm\times nm},$ the subsequent assertions are satisfied:
                                            \begin{enumerate}[(i)]
     \item $Q$ is positive, symmetric, block-Toeplitz with Toeplitz-block (BTTP), and positive definite.
     \item $ \|Q\|_{\infty} = \|Q\|_1<\frac{(1+\lambda_1)(1+\lambda_2)}{(1-\lambda_1)(1-\lambda_2)}.$
     \item $Q$ is invertible and satisfy 
       $\|Q^{-1}\|_{\infty}=\|Q^{-1}\|_1\leq \frac{(1+\lambda_1)(1+\lambda_2)}{(1-\lambda_1)(1-\lambda_2)}.$
     \item $\kappa_{\infty}(Q)=\kappa_1(Q)<\frac{(1+\lambda_1)^2(1+\lambda_2)^2}{(1-\lambda_1)^2(1-\lambda_2)^2}.$
 \end{enumerate}
\end{theorem}

Building upon the foundational work in \cite{bai2024, bai2025}, this study advances the analysis of upper bounds by employing alternative representations and matrix decompositions, including those of their inverses. In addition, the article investigates various spectral and numerical range properties associated with one- and two-dimensional Wasserstein-$1$ metric matrices. The central objectives of this work are as follows:\\
(i) to establish both lower and upper bounds for the $1,2,~\text{and}~\infty$-norms of the matrices;\\
(ii) to refine the existing upper bounds for $\|Q\|_1$ and the condition number $\kappa_1(Q)$;\\
(iii) to determine regions enclosing the eigenvalues and numerical ranges of these matrices; \\
(iv) to decompose the one-dimensional Wasserstein-$1$ metric matrix and  its Hadamard inverse as the product of two matrices,\\
(v) to find any other decomposition and\\
(vi) to find the determinant of Wasserstein-$1$ metric matrices.

                              \subsection{Contributions}
The key contributions of this article are summarized below. 
                                        \begin{itemize}
    \item Some upper and lower bounds on the $1,2,~\text{and}~\infty$-norm of the one and two-dimensional Wasserstein-$1$ metric matrices, their inverses, and their condition numbers are provided which are much sharper than those provided in Theorem \ref{thm:wass one dim norm bounds} and Theorem \ref{thm:wass two dim norm bounds} for the $1$ and $\infty$-norms. The results are further demonstrated through graphical representations, with the computed bounds tabulated for matrices of varying sizes using MATLAB. 
    \item In addition, some regions are estimated for the eigenvalues and numerical ranges of the Wasserstein-$1$ metric matrices of dimensions one and two along with the determinant. 
    \item Different ways to decompose the one-dimensional Wasserstein-$1$ metric matrix are explored through the notion of the Cayley transform of a matrix and the Hadamard product of matrices.  Further, a decomposition of the Hadamard inverse of the Wasser-
stein matrix is obtained. A way to find the condition number through the numerical radius of the matrix is also established. 
\end{itemize}
                        \section{Background results}\label{sec:Preliminaries}
This section recalls some background knowledge and notations to build the main structure of the article. $\mathbb{R}^n (\mathbb{C}^n)$ and $\mathbb{R}^{n\times n}(\mathbb{C}^{n\times n})$ denote the set of all $n$-dimensional real (complex) vectors and the set of all $n\times n$ real (complex) matrices, respectively. We use $\|\cdot\|_1,~\|\cdot\|_2,$ and $\|\cdot\|_{\infty}$ to denote the $1$-norm, $2$-norm, and $\infty$-norm of a matrix, respectively. $(\cdot)^T((\cdot)^*)$ represents the transpose (conjugate transpose) of a vector or a matrix. For a nonsingular matrix $A$, $\kappa_{\#}(A)=\|A\|_{\#}\|A^{-1}\|_{\#}$ indicates the condition number of $A,$ where $\#\in \{1,2,\infty\}.$ $\sigma(A)$ and $\rho(A)$ stand for the set of all eigenvalues of $A$,  the spectral radius of $A.$ 
We recall some results related to the eigenvalues of a matrix, which will be used to prove our main results.
                      \begin{theorem}(Theorem 6.1.1, \cite{horn1985})\label{thm:spectrum in gersgorin}\\
Let $A \in \mathbb{C}^{n\times n}.$ Then, 
$$\sigma(A) \subset G(A)= \displaystyle \bigcup_{i=1}^{n}\{z \in \mathbb{C}:\; |z-a_{ii}|\leq R_i'(A)\},$$
    \text{where} $R_i'(A)=\displaystyle \sum_{\substack {j=1 \\ j\neq i}}^{n} a_{ij}$ for any $i=1,2,\dots,n.$
\end{theorem}
A nonnegative (positive) matrix means that its each entry is nonnegative (positive). 
                             \begin{theorem}(Lemma 8.4.2, \cite{horn1985})\label{thm:spctrl rds}\\
If $A \in \mathbb{R}^{n\times n}$ is nonnegative, then $\rho(I+A)=1+\rho(A).$   
\end{theorem}
                           \begin{theorem}(Problem 8.3.P10, \cite{horn1985})\label{thm:spctrl rds1}\\
If $A \in \mathbb{R}^{n\times n}$ is nonnegative, then $\rho(\text{Re}(A))=\rho(\frac{A+A^T}{2})\geq \rho(A).$
\end{theorem}
\begin{proof}
Since $A$ is nonnegative, $\rho(A)$ is an eigenvalue of $A$ and there exists a nonnegative nonzero unit vector such that $Ax=\rho(A)x,$ i.e., $\rho(A)=x^TAx$ (using Theorem 8.3.1, \cite{horn1985}).
Since $\frac{A+A^T}{2}$ is symmetric, $\rho(\frac{A+A^T}{2})=\displaystyle\max_{\|y\|_2=1}y^T\Big(\frac{A+A^T}{2}\Big)y.$
Therefore, \begin{align*}
 \rho\Big(\frac{A+A^T}{2}\Big)&\geq x^T\Big(\frac{A+A^T}{2}\Big)x
 \end{align*}
 \begin{align*}
 &=\frac{x^TAx+x^TA^Tx}{2}\\
  &=\frac{x^TAx+(Ax)^Tx}{2}\\
 &=\frac{x^T(\rho(A)x)+(\rho(A)x)^Tx}{2}\\
 &=\frac{2 \rho(A)x^Tx}{2}\\
 &=\rho(A).
\end{align*}
\end{proof}
                          \begin{theorem}(Corollary 8.1.20, \cite{horn1985})\label{thm:spctrl rds3}\\
If $A \in \mathbb{R}^{n\times n}$ is nonnegative, then  $\rho(A[\beta])\leq\rho(A),$ where $\rho(A[\beta])$ is any principal submatrix of $A.$  
\end{theorem}
The following theorem is a part of the famous Perron's theorem.
                        \begin{theorem}(Theorem 8.2.8 (Perron), \cite{horn1985})\label{thm:spctrl rds4}\\
If $A \in \mathbb{R}^{n\times n}$ is positive, then   $\rho(A)\in\sigma(A).$   
\end{theorem}
For a matrix $A \in \mathbb{C}^{n\times n},$ the numerical range is denoted by $W(A)$ and defined as 
$W(A)=\{x^*Ax:\,x\in\mathbb{C}^{n},~\|x\|_2=x^{*}x=1\}.$ 
And, the numerical radius $\omega(A)$ is the largest absolute value among the elements in the numerical range. Some properties of the numerical radius are recalled here:
                                     \begin{enumerate}
    \item [(P1)] $\omega(A)=\omega(A^T).$ \label{P1}
    \item [(P2)] $\omega(I+A)=1+\omega(A).$
    \item [(P3)] $\omega(A+B)\leq \omega(A)+\omega(B).$
    \item [(P4)] $\rho(A)\leq \omega(A).$
    \item [(P5)] $\frac{\|A\|_2}{2}\leq \omega(A) \leq \|A\|_2.$
\end{enumerate}
For the proof of all these properties, we refer to \cite{gustafson1997, horn1991}.
                     \begin{theorem}(Theorem 1, \cite{haagerup1992})\label{thm:nilp NRd}\\
Let $A \in \mathbb{C}^{n\times n}$ be such that $A^m=O$ for some $m\geq 2.$ Then,
$\omega(A)\leq \|A\|_2 \cos\Big(\frac{\pi}{m+1}\Big).$
\end{theorem}       
Now, we recall a result related to the $2$-norm of the Hadamard product of two matrices.
\begin{theorem}(Theorem 5.5.3, \cite{horn1991})\label{thm:had norm}\\
For any two matrices $A,B\in  \mathbb{C}^{m\times n}$, we have 
$$\|A \circ B\|_2\leq r_1(A) c_1(B),$$ 
where $r_1(A)=\displaystyle\max_{1\leq i\leq m}\Big(\sum_{j=1}^{n}|a_{ij}|^2\Big)^{1/2}$ and $c_1(B)=\displaystyle\max_{1\leq j\leq n}\Big(\sum_{i=1}^{m}|a_{ij}|^2\Big)^{1/2}.$  
\end{theorem}
                         \section{Main Results}\label{sec:main results}
                        
                        \subsection{One-dimensional Wasserstein-1 metric matrix}\label{sec:1 dim wass mat}
 In 2024, Bai \cite{bai2024} estimated an upper bound for the $1$ and $\infty$-norms of a one-dimensional Wasserstein-$1$ metric matrix (see Theorem \ref{thm:wass one dim norm bounds}(ii)). In 2025, Bai \cite{bai2025} also found an upper bound for the condition number of a one-dimensional Wasserstein-$1$ metric matrix (see Theorem \ref{thm:wass one dim norm bounds}(iv)). The next result provides a tight upper bound for $\|Q\|_1$, tight lower bound for $\|Q^{-1}\|_1$,  upper and lower bounds for $\|Q^{-1}\|_2$, tight upper and lower bounds for $\kappa_1(Q)$, and upper and lower bounds for $\kappa_2(Q)$, among other results.
                                      \begin{theorem}\label{thm:one dim wass}
For the one-dimensional Wasserstein-$1$ metric matrix $Q\in \mathbb{R}^{n \times n}$ defined in \eqref{eq:exp Q L}, the following hold: 
                                    \begin{enumerate}[(i)]
\item If $\lambda\leq \frac{1}{3}$, then for any eigenvalue $\alpha$ of $Q$, $0<\alpha< 2.$  
\item For any eigenvalue $\alpha$ of $Q$, $0<\alpha\leq 1+2\|L\|_2\cos ({\frac{\pi}{n+1}}).$
\item $\rho(Q)  \geq 1+\lambda.$
\item $W(Q)\subset (0,\;\frac{(1+\lambda)(1-\lambda^{n-1})}{(1-\lambda)}].$
\item $\frac{1}{(1+\lambda)^2}\leq \|Q\|_{\infty}=\|Q\|_1\leq 1+\frac{2 \lambda (1-\lambda^{n-1})}{(1-\lambda)}.$
\item $\frac{1}{(1+\lambda)^2}\leq \|Q\|_2 \leq \min\{1+2\|L\|_2, \frac{3-\lambda}{1-\lambda},\frac{2(1+\lambda)(1-\lambda^{n-1})}{(1-\lambda)}\}.$
\item $\|Q^{-1}\|_{\infty}=\|Q^{-1}\|_1\geq \frac{(1-\lambda)}{(1+\lambda)(1-\lambda^n)^2}.$
\item $\frac{(1-\lambda)}{(1+\lambda)}\leq \|Q^{-1}\|_2\leq \frac{(1+\lambda)}{(1-\lambda)}.$
\item $\frac{(1-\lambda)}{(1+\lambda)^3(1-\lambda^n)^2}\leq \kappa_{\infty}(Q)=\kappa_1(Q)\leq \frac{(1+\lambda)(1-\lambda+2\lambda(1-\lambda^{n-1}))}{(1-\lambda)^2}.$
\item $\frac{(1-\lambda)}{(1+\lambda)^3}\leq \kappa_2(Q)\leq \frac{(1+\lambda)}{(1-\lambda)}\cdot \min\{1+2\|L\|_2, \frac{3-\lambda}{1-\lambda},\frac{2(1+\lambda)(1-\lambda^{n-1})}{(1-\lambda)}\}.$      
\end{enumerate}    
\end{theorem}
                                         \begin{proof}
                                           \begin{enumerate}[(i)]
    \item Using Theorem \ref{thm:spectrum in gersgorin}, we have 
    $$G(Q)= \displaystyle \bigcup_{i=1}^{n}\{z \in \mathbb{C}:\; |z-q_{ii}|\leq R_i'(Q)\}.$$
    For $\lambda \leq \frac{1}{3}$, we have (see proof of Theorem 2.1(i), \cite{bai2024}) $R_i'(Q)<1.$ So, 
    $$G(Q)=\displaystyle \bigcup_{i=1}^{n}\{z \in \mathbb{C}:\; |z-1|\leq R_i'(Q)<1\}=\{z\in \mathbb{C}:\; |z-1|<1\},$$
    which is an open disc in $\mathbb{R}^2$ plane with centered at 1 and of radius 2.
    Since $\sigma(Q)\subset G(Q)$ (see Theorem \ref{thm:spectrum in gersgorin}),
     for any eigenvalue $\alpha$ of $Q$, we have $0<\alpha< 2.$ 
    \item Using \eqref{eq:exp Q L}, we have
                                        \begin{align*}
        \omega(Q)&=\omega(I+L+L^T)\\
        &=\omega(I)+\omega(L+L^T)~(\text{using (P2)})\\
        &\leq 1+ \omega(L)+\omega(L^T)~(\text{using (P3)})\\
        &=1+2\omega(L)~~~~~~~~~~~~(\text{using (P1)})\\
        &\leq 1+2~\|L\|_2\cos\Big({\frac{\pi}{n+1}}\Big)~~(\text{using Theorem}~\ref{thm:nilp NRd}).
    \end{align*}
    Since $\rho(Q)\leq \omega(Q)$ (using (P4)), for any $\alpha \in \sigma(Q)$, we get
   $0<\alpha\leq 1+2~\|L\|_2\cos({\frac{\pi}{n+1}}).$  
   \item We have $Q=I+L+L^T=I+2\text{Re}(L).$ So, 
                                           \begin{align*}
       \rho(Q)&=\rho (I+2\text{Re}(L)) \\
       &= 1+2\rho (\text{Re}(L))~(\text{using Theorem}~\ref{thm:spctrl rds})\\
       &\geq 1 +2 \rho(L)~(\text{using Theorem}~\ref{thm:spctrl rds1})\\
       &= 1 + 0\\
       &=1.       
   \end{align*}
 Using Theorem \ref{thm:spctrl rds4}, $\rho(Q)\in \sigma(Q)$.
Now, for $\beta=\{1,2\}$, we have
   $$Q[\beta]=\begin{bmatrix}
       1& \lambda\\
       \lambda & 1
   \end{bmatrix}.$$ Therefore, $\sigma(Q[\beta])=\{1\pm \lambda\}.$ 
By applying Theorem \ref{thm:spctrl rds3}, we get 
   $$\rho(Q[\beta])=1+\lambda\leq \rho(Q).$$
   Thus, $\rho(Q)\geq \max \{1, 1+\lambda\}=1+\lambda.$
   \item Let $x \in \mathbb{C}^n$ be a unit vector. Then, 
                                              \begin{align*}
       x^*Qx&= \displaystyle\sum_{i=1}^{n}|x_i|^2+\lambda \displaystyle \sum_{i=1}^{n-1}(\bar{x}_ix_{i+1}+\bar{x}_{i+1}x_{i})+\lambda^2 \displaystyle \sum_{i=1}^{n-2}(\bar{x}_ix_{i+2}+\bar{x}_{i+2}x_{i})+\dots+\\
       &~~~~~\lambda^{n-2}\displaystyle \sum_{i=1}^{2}(\bar{x}_ix_{i+(n-2)}+\bar{x}_{i+(n-2)}x_{i})+\lambda^{n-1}(\bar{x}_1x_n+\bar{x}_nx_1)
             \end{align*}
 \begin{align*}
         &=1+2\lambda \displaystyle \sum_{i=1}^{n-1}\text{Re}(\bar{x}_ix_{i+1})+2\lambda^2 \displaystyle \sum_{i=1}^{n-2}\text{Re}(\bar{x}_ix_{i+2})+\dots+\\
       &~~~~~2\lambda^{n-2}\displaystyle \sum_{i=1}^{2}\text{Re}(\bar{x}_ix_{i+(n-2)})+2\lambda^{n-1}\text{Re}(\bar{x}_1x_n)~~\Big(\because \frac{z+\bar{z}}{2}=\text{Re}(z)\Big) \\      
 &\leq 1+2\lambda \displaystyle \sum_{i=1}^{n-1}|\bar{x}_ix_{i+1}|+2\lambda^2 \displaystyle \sum_{i=1}^{n-2}|\bar{x}_ix_{i+2}|+\dots+\\
   &~~~~~2\lambda^{n-2}\displaystyle \sum_{i=1}^{2}|\bar{x}_ix_{i+(n-2)}|+2\lambda^{n-1}|\bar{x}_1x_n| ~~(\because \text{Re}(z)\leq|z|)\\
        &=1+2\lambda \displaystyle \sum_{i=1}^{n-1}|x_i||x_{i+1}|+2\lambda^2 \displaystyle \sum_{i=1}^{n-2}|x_i||x_{i+2}|+\dots+\\
       &~~~~~2\lambda^{n-2}\displaystyle \sum_{i=1}^{2}|x_i||x_{i+(n-2)}|+2\lambda^{n-1}|x_1||x_n|~~(\because~|z_1z_2|=|z_1||z_2|~\text{and}~|z|=|\bar{z}|) \\         &\leq 1+\lambda \displaystyle \sum_{i=1}^{n-1}(|x_i|^2+|x_{i+1}|^2)+\lambda^2 \displaystyle \sum_{i=1}^{n-2}(|x_i|^2+|x_{i+2}|^2)+\dots+\\
       &~~~~~\lambda^{n-2}\displaystyle \sum_{i=1}^{2}(|x_i|^2+|x_{i+(n-2)}|^2)+\lambda^{n-1}(|x_1|^2+|x_n|^2)~~(\because~\text{AM$\geq$GM})\\
        &=1+\lambda \Big(\displaystyle \sum_{i=1}^{n-1}|x_i|^2+\displaystyle \sum_{i=1}^{n-1}|x_{i+1}|^2\Big)+\lambda^2 \Big(\displaystyle \sum_{i=1}^{n-2}|x_i|^2+\displaystyle \sum_{i=1}^{n-2}|x_{i+2}|^2\Big)+\dots+\\
       &~~~~~\lambda^{n-2}\Big(\displaystyle \sum_{i=1}^{2}|x_i|^2+\displaystyle \sum_{i=1}^{2}|x_{i+(n-2)}|^2\Big)+\lambda^{n-1}(|x_1|^2+|x_n|^2)\\ 
       &\leq 1+\lambda(1+1)+\lambda^2(1+1)+\dots+\lambda^{n-2}(1+1)+\lambda^{n-1}\cdot 1~~(\because x~\text{is a unit vector})\\
       &=(1+\lambda+\lambda^2+\dots+\lambda^{n-1})+(\lambda+\lambda^2+\dots+\lambda^{n-2})\\
        &=\frac{1\cdot (1-\lambda^{n})}{1-\lambda}+\frac{\lambda \cdot (1-\lambda^{n-2})}{1-\lambda}\\
        &=\frac{1-\lambda^{n}+\lambda-\lambda^{n-1}}{1-\lambda}\\
       &= \frac{(1+\lambda)(1-\lambda^{n-1})}{1-\lambda}\\
      &<\frac{1+\lambda}{1-\lambda}~~(\because~\lambda\in (0,1)).
   \end{align*}
   Thus, the required assertion follows.
\end{enumerate}
 \end{proof}
                                                 \begin{remark}
 The inequality (v) improves the upper bound of the norm inequality 
 $$\|Q\|_{\infty}=\|Q\|_1 <\frac{(1+\lambda)}{(1-\lambda)}$$ proved in Theorem \ref{thm:wass one dim norm bounds}(ii). It can be observed from the following relation:\\
 $$1+\frac{2\lambda(1-\lambda^{n-1})}{(1-\lambda)}<1+\frac{2\lambda}{(1-\lambda)}=\frac{(1+\lambda)}{(1-\lambda)}.$$
 Also, from the same observation, we can see that the upper bound in (ix) refines the upper bound in the inequality
 $$\kappa_{\infty}(Q)=\kappa_1(Q)<\frac{(1+\lambda)^2}{(1-\lambda)^2}$$
 proved in Theorem \ref{thm:wass one dim norm bounds}(iv). 
\end{remark}
Using Theorem \ref{thm:had norm}, one can also find some bounds for the $2$-norm of the Wassetrstein matrix $Q$ and the matrix  $L.$ 
                                         \begin{theorem}\label{thm:had norm wass}
For the matrices $Q$ and $L$, we have 
                                         \begin{enumerate}[(i)]
    \item $\|Q\|_2\leq n.$
    \item $\|L\|_2\leq \sqrt{(n-2+\lambda^2)\Big(1+\frac{\lambda^4-\lambda^{2n}}{1-\lambda^2}\Big)}.$
    \item $det(Q)=(1-\lambda^2)^{n-1}.$
\end{enumerate}
\end{theorem}
                                     \begin{proof}
                                     \begin{enumerate}[(i)]
        \item On breaking the matrix $Q$ as the  Hadamard product of two matrices $A$ and $B$, we have
        \begin{align*}
        Q&=\begin{bmatrix}
 1 & \lambda & \lambda^2 & \dots & \lambda ^{n-1}\\
     \lambda & 1 & \lambda & \dots & \lambda^{n-2}\\
     \lambda^2 & \lambda & 1 & \dots & \lambda^{n-3}\\
     \vdots & \vdots & \vdots & \dots & \vdots \\
     \lambda ^{n-1} & \lambda ^{n-2} &\lambda ^{n-3} & \dots & 1
     \end{bmatrix}\\
     &=\begin{bmatrix}
 1 & 1 & 1 & \dots & 1\\
     \lambda & 1 & 1 & \dots &1\\
     \lambda^2 & \lambda & 1 & \dots &1\\
     \vdots & \vdots & \vdots & \dots & \vdots \\
     \lambda ^{n-1} & \lambda ^{n-2} &\lambda ^{n-3} & \dots & 1
     \end{bmatrix}\circ \begin{bmatrix}
 1 & \lambda & \lambda^2 & \dots & \lambda ^{n-1}\\
     1 & 1 & \lambda & \dots & \lambda^{n-2}\\
    1 & 1& 1 & \dots & \lambda^{n-3}\\
     \vdots & \vdots & \vdots & \dots & \vdots \\
    1 & 1 &1 & \dots & 1
     \end{bmatrix}
     =A\circ B=A\circ A^T,
        \end{align*}
and using Theorem \ref{thm:had norm}, we obtain
$$\|Q\|_2\leq r_1(A)c_1(B)=\sqrt{n}\cdot \sqrt{n}=n.$$
\item On breaking the matrix $L$ as the Hadamard product of two matrices $C$ and $D$, we get
                                   \begin{align*}
         L&=\begin{bmatrix}
     0 & 0 & 0 & \dots & 0 & 0\\
     \lambda & 0 & 0 & \dots & 0& 0\\
     \lambda^2 & \lambda & 0 & \dots & 0& 0\\
     \vdots & \vdots & \vdots & \ddots & \vdots& \vdots \\
     \lambda ^{n-1} & \lambda ^{n-2} &\lambda ^{n-3} & \dots & \lambda & 0
     \end{bmatrix}\\
     &=\begin{bmatrix}
 0 & 0 & 0 & \dots & 0& 0\\
     \lambda & 0 & 0 & \dots &0&0\\
     1& \lambda & 0 & \dots &0&0\\
     \vdots & \vdots & \vdots & \ddots &\vdots& \vdots \\
     1 & 1 &1 & \dots & \lambda & 0
     \end{bmatrix} \circ \begin{bmatrix}
 0 & 0 & 0 & \dots & 0&0\\
     1 & 0 & 0 & \dots &0&0\\
     \lambda^2 & 1 & 0 & \dots &0&0\\
     \vdots & \vdots & \vdots & \ddots & \vdots  & \vdots \\
     \lambda ^{n-1} & \lambda ^{n-2} &\lambda ^{n-3} & \dots &1& 0
     \end{bmatrix}=C\circ D
     \end{align*}
     Now, $r_1(C)=\sqrt{n-2+\lambda^2}$ and 
                                          \begin{align*}
         c_1(D)&=\sqrt{1+((\lambda^2)^2+(\lambda^3)^2+\dots+(\lambda^{n-1})^2}\\
         &=\sqrt{1+\frac{\lambda^4(1-\lambda^{2(n-2)})}{1-\lambda^2}}\\
         &= \sqrt{1+\frac{\lambda^4-\lambda^{2n}}{1-\lambda^2}}.
     \end{align*}
    Therefore, using Theorem \ref{thm:had norm}, we obtain
$$\|L\|_2\leq r_1(C)c_1(D)=\sqrt{(n-2+\lambda^2)\Big(1+\frac{\lambda^4-\lambda^{2n}}{1-\lambda^2}\Big)}.$$
\item Using the third elementary row operations, we get $det(Q)=(1-\lambda^2)^{n-1}.$
    \end{enumerate}
\end{proof}
\begin{remark}
    The proof of the first part also shows that the Wasserstein matrix $Q$ can also be decomposed as the Hadamard product of a matrix and its transpose. This may be helpful for studying several other properties of the Wasserstein matrix $Q$.
    
\end{remark} 
                      
                                             \begin{remark}
                                              \begin{enumerate}    
    \item We know that a matrix is symmetric if and only if its Cayley transform is symmetric \cite{mondal2024}, where the Cayley transform of a matrix $A$ is defined as $F=C(A)=(I+A)^{-1}(I-A)$. More on Cayley transform can be found in  \cite{mondal2024, verma2024, verma2025}.
It is also known that $A=C(F)=(I + F)^{-1}(I-F)$. Since $Q$ is symmetric, we have another new decomposition of $Q$ as: $Q=(I + F)^{-1}(I-F),$ where $F=C(Q)$, $(I + F)^{-1}$ and $(I-F)$ are symmetric. 
\item We also know that for a positive definite matrix $A$, the Cayley transform $C(A)$ of $A$ is positive definite if and only if $\sigma(A)\subset (0,1)$ \cite{mondal2024}. From Theorem \ref{thm:one dim wass}(iii), it can be observed that the $C(Q)$ can never be positive definite.
\end{enumerate}    
\end{remark} 
                                           \begin{remark}
We can also find bounds on the condition number using Theorem 7, \cite{chien2020}. However, to do so, we need to estimate $\omega(Q^{-1}).$ But, using Theorem \ref{thm:one dim wass}(iv) \& (viii), we can also obtain
$1=\omega(I)=\omega(QQ^{-1})\leq \|QQ^{-1}\|_2\leq\|Q\|_2\|Q^{-1}\|_2=\kappa_2(Q)\leq 2 \omega(Q) \|Q^{-1}\|_2\leq \frac{2(1+\lambda)^2(1-\lambda^{n-1})}{(1-\lambda)^2}<\frac{2(1+\lambda)^2}{(1-\lambda)^2}.$
\end{remark}
 Inverting a  matrix of order $n$ takes 
$O(n^3)$ operations while the
Hadamard inverse is just $O(n^2)$.
It doesn't require the matrix to be invertible,  and has applications in signal processing, Optimization, etc.  The  Hadamard inverse exists for any square matrix $A$ with non-zero entries, even though the usual matrix inverse doesn't exist.  Now, we provide the norm bounds for the Hadamard inverse 
$$Q^{\circ -1}=\begin{bmatrix}
 1 & \frac{1}{\lambda} & \frac{1}{\lambda^2} & \dots & \frac{1}{\lambda^{n-1}}\\
     \frac{1}{\lambda} & 1 & \frac{1}{\lambda} & \dots & \frac{1}{\lambda^{n-2}}\\
     \frac{1}{\lambda^2} & \frac{1}{\lambda} & 1 & \dots & \frac{1}{\lambda^{n-3}}\\
     \vdots & \vdots & \vdots & \ddots & \vdots \\
     \frac{1}{\lambda^{n-1}} & \frac{1}{\lambda^{n-2}} &\frac{1}{\lambda^{n-3}} & \dots & 1
     \end{bmatrix}$$ of the one dimensional Wasserstein$-1$ matrix $Q.$
     \begin{theorem}\label{thm:hadamard inv norm bound}
For the one-dimensional Wasserstein-$1$ metric matrix $Q\in \mathbb{R}^{n \times n}$ defined in \eqref{eq:exp Q L}, the following hold: 
\begin{enumerate}[(i)]
\item $\|Q^{\circ -1}\|_1=\|Q^{\circ -1}\|_{\infty}=\frac{1}{\lambda^{n-1}}\Big(\frac{1-\lambda^n}{1-\lambda}\Big);$
\item $\|Q^{\circ -1}\|_2\leq \frac{1}{\lambda^{2(n-1)}}\Big(\frac{1-\lambda^{2n}}{1-\lambda^2}\Big).$
\end{enumerate}
\end{theorem}
\begin{proof}
\begin{enumerate}[(i)]
\item By the definition of the $1$-norm of a matrix, we have 
\begin{align*}
\|Q^{\circ -1}\|_1&=1+\frac{1}{\lambda}+\frac{1}{\lambda^2}+\dots+ \frac{1}{\lambda^{n-1}}  \\
&=\frac{1}{\lambda^{n-1}}\Big(\frac{1-\lambda^n}{1-\lambda}\Big).
\end{align*}

\item We can write the Hadamard inverse $Q^{\circ -1}$ as a product of two matrices as follows:
$$Q^{\circ -1}=\begin{bmatrix}
 1 & 1 & 1 & \dots & 1\\
     \frac{1}{\lambda} & 1 & 1 & \dots & 1\\
     \frac{1}{\lambda^2} & \frac{1}{\lambda} & 1 & \dots & 1\\
     \vdots & \vdots & \vdots & \ddots & \vdots \\
     \frac{1}{\lambda^{n-1}} & \frac{1}{\lambda^{n-2}} &\frac{1}{\lambda^{n-3}} & \dots & 1
     \end{bmatrix}\circ \begin{bmatrix}
 1 & \frac{1}{\lambda} & \frac{1}{\lambda^2} & \dots & \frac{1}{\lambda^{n-1}}\\
    1 & 1 & \frac{1}{\lambda} & \dots & \frac{1}{\lambda^{n-2}}\\
     1 & 1 & 1 & \dots & \frac{1}{\lambda^{n-3}}\\
     \vdots & \vdots & \vdots & \ddots & \vdots \\
     1 & 1 &1 & \dots & 1
     \end{bmatrix}=A\circ B.$$
Using Theorem \ref{thm:had norm}, we have
$\|Q^{\circ -1}\|_2\leq r_1(A) c_1(B).$ 
Now, 
\begin{align*}
r_1(A)=c_1(B)&=\sqrt{1+\Big(\frac{1}{\lambda}\Big)^2+\Big(\frac{1}{\lambda^2}\Big)^2+\dots+\Big(\frac{1}{\lambda^{n-1}}\Big)^2}   \\
&=\sqrt{\frac{1}{\lambda^{2(n-1)}}\Big(\frac{1-\lambda^{2n}}{1-\lambda^2}\Big)}.
\end{align*}
Thus, we get the desired inequality for the $2$-norm of $Q^{\circ -1}$.
\end{enumerate}
\end{proof}

\begin{remark}
From the proof of (ii) of the above theorem, it is clear that the Hadamard inverse of $Q$ can now be decomposed as the Hadamard product of two matrices $A$ and $B$.
\end{remark}
In the following table, the upper bounds for the spectral norm of the Hadamard inverse and usual inverse are compared. 
\begin{table}[h!]
\begin{center}
\caption{Comparison table for the upper bounds of the spectral norm of the Hadamard inverse $Q^{\circ -1}$ and the usual inverse $Q^{-1}$ of the matrix $Q$ for $\lambda=0.5$}\label{table:2-norm comp of had and usual inv of Q for lambda=0.5}
\begin{tabular}{ |p{1cm}|p{4cm}|p{4cm}|p{4cm}|}
\hline
& 	 &  &  \\
{\bf n } & 	{\bf $y_1=\frac{1}{\lambda^{2(n-1)}}\Big(\frac{1-\lambda^{2n}}{1-\lambda^2}\Big)$} & $y_2= \frac{(1+\lambda)}{(1-\lambda)}$ & $y_1-y_2$\\
& 	 &  &  \\
 &(the upper bound for the Hadamard inverse) 	 & (the upper bound for the usual inverse)  &  \\
\hline
1 	& 1	& 3&  -2\\
\hline
2	& 5	& 3 & 2  \\
\hline
3 	& 21	& 3 & 18\\
\hline
4 	& 85	&  3 &  82\\
\hline
5	& 341	& 3 &  338 \\
\hline
10	&  349525	& 3 &   349522 \\
\hline
\end{tabular}
\end{center}
\end{table}
                        \subsection{Two-dimensional Wasserstein-1 metric matrix}\label{sec:2 dim wass mat}
Before proving the next result, we recall some properties about the Kronecker product of two matrices $A\in \mathbb{R}^{n\times n}$ and $B\in \mathbb{R}^{m\times m}$:
                                          \begin{enumerate}
\item [(R1)] $\|A \otimes B\|=\|A\|\|B\|.$
\item [(R2)] $\sigma(A \otimes B)=\{\alpha_i \beta_j:i=1,2,\dots,n,j=1,2,\dots,m, \alpha_i \in \sigma(A),\beta_j \in \sigma(B) \}$ (see Theorem 5.4-1, \cite{gustafson1997}).
\item [(R3)] If $A$ is a normal matrix, then $W(A \otimes B)=\text{Co}(W(A)W(B)),$ where $\text{Co}(S)$ is the convex hull of the set $S$ (see Theorem 5.4-3, \cite{gustafson1997}).
\item [(R4)] If $A$ and $B$ are invertible, then $(A \otimes B)^{-1}=A^{-1}\otimes B^{-1}.$ 
\item [(R5)] $(A\otimes B)^{\circ -1} =A^{\circ -1}\otimes B^{\circ -1}$ for any square $A$ and $B$ with non zero entries. 
\end{enumerate}
In 2024, Bai \cite{bai2024} found an upper bound for the $1$ and $\infty$-norms of the two-dimensional Wasserstein-$1$ metric matrices (see Theorem \ref{thm:wass two dim norm bounds}(ii)). In 2025, Bai \cite{bai2025} also found an upper bound for the condition number of a two-dimensional Wasserstein-$1$ metric matrix (see Theorem \ref{thm:wass two dim norm bounds}(iv)). In our next main result, we refine these upper bounds with some other results. 
                                                  \begin{theorem}
For the two dimensional Wasserstein-$1$ metric matrix $Q\in \mathbb{R}^{nm \times nm}$, the following are true: 
                                          \begin{enumerate}[(i)]
         \item If $\lambda_1,\lambda_2\leq \frac{1}{3}$, then for any eigenvalue $\alpha$ of $Q$, $0<\alpha< 4.$  
         \item For any eigenvalue $\alpha$ of $Q$, $0<\alpha\leq(1+2~\|L_{[1]}\|_2\cos ({\frac{\pi}{n+1}}))(1+2~\|L_{[2]}\|_2\cos ({\frac{\pi}{m+1}})).$  
         \item $\rho(Q)  \geq (1+\lambda_1)(1+\lambda_2).$
         \item $W(Q)\subset (0,\;\frac{(1+\lambda_1)(1+\lambda_2)(1-\lambda_1^{n-1})(1-\lambda_2^{n-1})}{(1-\lambda_1)(1-\lambda_2)}]
         .$
         \item $\frac{1}{(1+\lambda_1)^2(1+\lambda_2)^2}\leq \|Q\|_1= \|Q\|_{\infty}\leq (1+\frac{2 \lambda_1 (1-\lambda_1^{n-1})}{(1-\lambda_1)})(1+\frac{2 \lambda_2 (1-\lambda_2^{n-1})}{(1-\lambda_2)}).$
         \item $\frac{1}{(1+\lambda_1)^2(1+\lambda_2)^2}\leq \|Q\|_2\leq m_1m_2,$
         where $m_k=\min \{1+2\|L_{[k]}\|_2, \frac{3-\lambda_k}{1-\lambda_k},\frac{2(1+\lambda_k)(1-\lambda_k^{n-1})}{(1-\lambda_k)}\}, k=1,2.$
         \item $\|Q^{-1}\|_1=\|Q^{-1}\|_{\infty}\geq \frac{(1-\lambda_1)(1-\lambda_2)}{(1+\lambda_1)(1+\lambda_2)(1-\lambda_1^n)^2(1-\lambda_2^n)^2}.$
         \item $\frac{(1-\lambda_1)(1-\lambda_2)}{(1+\lambda_1)(1+\lambda_2)}\leq \|Q^{-1}\|_2\leq \frac{(1+\lambda_1)(1+\lambda_2)}{(1-\lambda_1)(1-\lambda_2)}.$
         \item $\frac{(1-\lambda_1)(1-\lambda_2)}{(1+\lambda_1)^3(1+\lambda_2)^3(1-\lambda_1^n)^2(1-\lambda_2^n)^2}\leq \kappa_1(Q)=\kappa_{\infty}(Q)\leq \frac{(1+\lambda_1)(1+\lambda_2)(1-\lambda_1+2\lambda_1(1-\lambda_1^{n-1}))(1-\lambda_2+2\lambda_2(1-\lambda_2^{n-1}))}{(1-\lambda_1)^2(1-\lambda_2)^2}.$
         \item $\frac{(1-\lambda_1)(1-\lambda_2)}{(1+\lambda_1)^3(1+\lambda_2)^3}\leq \kappa_2(Q)\leq m_1m_2\cdot\frac{(1+\lambda_1)(1+\lambda_2)}{(1-\lambda_1)(1-\lambda_2)}.$
     \end{enumerate}    
     \end{theorem}
                                           \begin{proof}
                                            \begin{enumerate}[(i)]
\item Let $\alpha \in \sigma(Q).$ Then, 
$\alpha=\beta_2\beta_1,$ for some
$\beta_1\in \sigma(Q_{[1]})$ and  $\beta_2\in \sigma(Q_{[2]}).$ Therefore, using  Theorem \ref{thm:one dim wass}(i) for $\beta_1$ and $\beta_2$, we have 
$0<\alpha< 4.$
\item Similar to (1), using Theorem \ref{thm:one dim wass}(ii), we can prove it.
\item From the above result (R2) of Kronecker product, we can get 
$$\rho(Q)=\rho(Q_{[2]}\otimes Q_{[1]})=\rho(Q_{[2]})\rho(Q_{[1]}).$$ 
Using Theorem \ref{thm:one dim wass}(iii), we have $\rho(Q)  \geq (1+\lambda_1)(1+\lambda_2).$
\item Since $Q_{[2]}$  is symmetric, $Q_{[2]}$ is normal. Therefore, we have 
$$W(Q)=\text{Co}(W(Q_{[2]})W(Q_{[1]}))$$
using the result (R3) of Kronecker product.
From Theorem \ref{thm:one dim wass}(iv), we have \\
$W(Q_{[1]})\subset (0,\;\frac{(1+\lambda_1)(1-\lambda_1^{n-1})}{1-\lambda_1}]$ and $W(Q_{[2]})\subset (0,\;\frac{(1+\lambda_2)(1-\lambda_2^{n-1})}{1-\lambda_2}].$ So,
                                              \begin{align*}
    \text{Co}(W(Q_{[2]})W(Q_{[1]}))&\subset \text{Co}\Big((0,\;\frac{(1+\lambda_1)(1-\lambda_1^{n-1})}{1-\lambda_1}]\cdot (0,\;\frac{(1+\lambda_2)(1-\lambda_2^{n-1})}{1-\lambda_2}]\Big)\\
   &= (0,\;\frac{(1+\lambda_1)(1+\lambda_2)(1-\lambda_1^{n-1})(1-\lambda_2^{n-1})}{(1-\lambda_1)(1-\lambda_2)}]. 
\end{align*}  
Thus,
$W(Q)\subset (0,\;\frac{(1+\lambda_1)(1+\lambda_2)(1-\lambda_1^{n-1})(1-\lambda_2^{n-1})}{(1-\lambda_1)(1-\lambda_2)}].$
\item Using the result (R1) for $\infty$-norm and Theorem \ref{thm:one dim wass}(v), we can get the desired result.
\item Using the result (R1) for $2$-norm and Theorem \ref{thm:one dim wass}(vi), we can get the desired result. 
\item Using (R4), we get
$$\|Q^{-1}\|_1
=\|Q_{[2]}^{-1}\|_1 \|Q_{[1]}^{-1}\|_1.$$
Now, using Theorem \ref{thm:one dim wass}(vii), we have the result.
\item Similar to (vii).\\
In the similar fashion, we can prove (ix) and (x).
\end{enumerate}  
\end{proof}
                                       \begin{remark}
Similar to the one-dimensional case, it can be observed that the upper bounds in the above inequalities (v) and (ix) sharpen the upper bounds provided in (ii) and (iv) part of Theorem \ref{thm:wass two dim norm bounds}, respectively.  
\end{remark}
                                            \begin{remark}
    Using the result (R1) and Theorem \ref{thm:had norm wass}, the following can be shown that for a two-dimensional Wasserstein-$1$ metric matrix $Q$.
    \begin{enumerate}[(i)]
        \item $\|Q\|_2\leq mn;$
        \item $det(Q)=(1-\lambda_1^2)^{m(n-1)}(1-\lambda_2^2)^{n(m-1)}.$
    \end{enumerate}
\end{remark}
\begin{remark}
Using the result (R5) and Theorem \ref{thm:hadamard inv norm bound} the following can be proved for a two-dimensional Wasserstein-$1$ metric matrix $Q$.
\begin{enumerate}[(i)]
    \item $\|Q^{\circ -1}\|_1=\|Q^{\circ -1}\|_{\infty}=\frac{1}{\lambda_1^{n-1}\lambda_2^{m-1}}\Big(\frac{1-\lambda_1^n}{1-\lambda_1}\Big)\Big(\frac{1-\lambda_2^m}{1-\lambda_2}\Big);$
\item $\|Q^{\circ -1}\|_2\leq \frac{1}{\lambda_1^{2(n-1)}\lambda_2^{2(m-1)}}\Big(\frac{1-\lambda_1^{2n}}{1-\lambda_1^2}\Big)\Big(\frac{1-\lambda_2^{2m}}{1-\lambda_2^2}\Big).$
\end{enumerate}
\end{remark}
                                \section*{Acknowledgements}
The first author acknowledges the support of the Council of Scientific and Industrial Research, India. The authors used MATLAB Online (R2024b) for all computational work presented in this paper.
                                 \section*{Conflict of Interest}
 The authors declare that there is no conflict of interest.

                                 \section*{Data Availability Statement}
Data sharing is not applicable to this article as no new data is analyzed in this study. 
 
                              \bibliographystyle{amsplain}

\begin{thebibliography}{10}

\bibitem{arjovsky2017}{Arjovsky, M.; Chintala, S.; Bottou, L.},
{\it Wasserstein generative adversarial networks}, Proceedings of the 34th International Conference on Machine Learning, in: Proceedings of Machine Learning Research, 70 (2017), 214-223.

\bibitem{bai2024}{Bai, Z.-Z.},
{\it The Wasserstein metric matrix and its computational property}, Linear Algebra Appl., 681 (2024), 150-186.

\bibitem{bai2025}{Bai, Z.-Z.},
{\it On bounds for norms and conditioning of Wasserstein metric matrix}, Appl. Math. Lett., 166 (2025), 109510.

\bibitem{blanchet2019}{Balanchet, J.; Kang, Y.; Murthy, K.},
{\it Robust Wasserstein profile inference and applications to machine learning}, 56 (2019), 830-857.

\bibitem{bhatia2019.EM}{Bhatia, R.; Jain, T.; Yangdo, L.},
{\it On the Bures–Wasserstein distance between positive definite
 matrices}, Expo. Math., 37 (2019), 165-191.

 \bibitem{bhatia2019.LAA}{Bhatia, R.; Jain, T.; Yangdo, L.},
{\it  Inequalities for the Wasserstein mean of positive 
definite matrices}, Linear Algebra  Appl., 576 (2019), 108-123.

\bibitem{chien2020}{Chien, M. T.}, {\it Numerical range of Moore-Penrose inverse matrices}, Mathematics, 8 (2020), 830.

\bibitem{cuturi2013}{Cuturi, M.},
{\it Sinkhorn distances: Lightspeed computation of optimal transportation distances,} Adv. 
Neural Inf. Process. Syst., 26 (2013), 2292-2300.

\bibitem{engquist2013arxiv}{Engquist, B.; Froese, B. D.},
{\it Application of the Wasserstein metric seismic signal}, (2013), arXiv:1311.4581v1.

\bibitem{gustafson1997}{Gustafson, K. E.; Rao, D. K. M.},
{\it Numerical range: The field of values of linear operators and matrices}, Springer-Verlag, New York, 1997.

\bibitem{haagerup1992}{Haagerup, U.; Harpe, P. D. L.},
{\it The numerical radius of a 
nilpotent operator on a Hilbert space}, Proc. Amer. Math. Soc., 115 (1992), 371-379. 

\bibitem{hayoung24}{Hayoung, C.; Sejong, K.; Yongdo, L.},
{\it Linearity of Cartan and Wasserstein means}, Linear Algebra  Appl., 681 (2024), 66-88. 

\bibitem{horn1985}{Horn, R. A.; Johnson, C. R.},
{\it Matrix Analysis},
Cambridge University Press, Cambridge, 1985.

\bibitem{horn1991}{Horn, R. A.; Johnson, C. R.},
{\it Topics in Matrix Analysis}, Cambridge University Press, Cambridge, 1991.

\bibitem{kolouri2016}{Kolouri, S.; Park, S. R.; Rohde, G. K.},
{\it The Radon cumulative distribution transform and its application 
to image classification}, IEEE Trans. Image Process., 25 (2016), 920-934.

\bibitem{kolouri2017}{Kolouri, S.; Park, S. R.; Thorpe, M.; Slepcev, D.; Rohde, G. K.},
{\it Optimal mass transport: Signal processing and machine-learning applications}, IEEE Signal Process. Mag., 34 (2017), 43-59.

\bibitem{liao2022}{Liao, Q.-C.; Chen, J.; Wang,  Z.-H.; Bai, B., Jin, S., Wu, H.},
{\it Fast Sinkhorn I: An O(N) algorithm for 
the Wasserstein-1 metric}, Commun. Math. Sci., 20 (2022), 2053-2067. 

\bibitem{liao2024}{Liao, Q.-C.; Wang,  Z.-H.; Chen, J.;  Bai, B., Jin, S., Wu, H.},
{\it Fast Sinkhorn II: Collinear triangular matrix 
and linear time accurate computation of optimal transport}, J. Sci. Comput., 98 (2024), 1-19. 

\bibitem{ling2007}{Ling, H.; Okada, K.},
{\it An efficient earth mover's distance algorithm for robust histogram comparison}, 29 (2007), 840-853.

\bibitem{mondal2024}{Mondal, S.; Sivakumar, K. C.; Tsatsomeros, M. J., {\it The Cayley transform of prevalent matrix classes,} Linear Algebra  Appl., 681 (2024), 1-20.}

\bibitem{rubner2000}{Rubner, Y.; Tomasi, C.; Guibas, L. J.},
{\it The earth mover’s distance as a metric for image retrieval}, J. Comput. Vis., 40 (2000), 99-121.

\bibitem{sinkhorn1967}{Sinkhorn, R.},
{\it Diagonal equivalence to matrices with prescribed row and column sums}, Am. Math. 
Mon., 74 (1967), 402-405.

\bibitem{sinkhorn1974}{Sinkhorn, R.},
{\it Diagonal equivalence to matrices with prescribed row and column sums. II}, Proc. Am. Math. Soc., 45 (1974), 195-198.

\bibitem{solomon2015}{Solomon, J.; De Goes, F.; Peyré, G.; Cuturi, M.; Butscher, A.; Nguyen, A.; Du, T.; Guibas, L.},
{\it Convolutional Wasserstein distances: Efficient optimal transportation on geometric domains}, ACM Trans. 
Graph., 34 (2015), 1-11.

\bibitem{verma2024}{Verma, T.; Mishra, D.; Tsatsomeros, M., {\it Cayley transform for Toeplitz and dual matrices,} Linear Algebra Appl., 703 (2024), 627-644.}

\bibitem{verma2025}{Verma, T.; Mishra, D.; Tsatsomeros, M.,} {\it Further results on Cayley transform of matrix classes}, Preprint.
\end{thebibliography}
                               
\end{document}